\documentclass[oneside,english]{amsart}
\usepackage[T1]{fontenc}
\usepackage[latin9]{inputenc}
\usepackage{amsthm}
\usepackage{amstext}
\usepackage{amssymb}
\usepackage{esint}
\usepackage{mathrsfs}
\usepackage{color,enumitem,graphicx}

\makeatletter

\numberwithin{equation}{section}
\numberwithin{figure}{section}
\theoremstyle{plain}
\newtheorem{thm}{\protect\theoremname}[section]
  \theoremstyle{plain}
  
  \theoremstyle{plain}
  
  \theoremstyle{plain}
  
  \theoremstyle{plain}
  \newtheorem{lem}[thm]{\protect\lemmaname}
	\theoremstyle{plain}
  \newtheorem{rem}[thm]{\protect\remarkname}
  \theoremstyle{definition}

\makeatother

\usepackage{babel}
  \providecommand{\corollaryname}{Corollary}
  \providecommand{\definitionname}{Definition}
  \providecommand{\lemmaname}{Lemma}
	\providecommand{\remarkname}{Remark}
  \providecommand{\propositionname}{Proposition}
  \providecommand{\examplename}{Example}
  \providecommand{\theoremname}{Theorem}

\DeclareMathOperator{\loc}{loc}

\DeclareMathOperator{\supp}{supp}

\DeclareMathOperator{\esssup}{esssup}
\DeclareMathOperator{\ess}{ess}
\DeclareMathOperator{\ACL}{ACL}

\begin{document}

\title[Sobolev homeomorphisms on homogeneous Lie groups]{Sobolev homeomorphisms and composition operators on homogeneous Lie groups}

\author{Alexander Ukhlov}

\begin{abstract}
In this article, we study Sobolev homeomorphisms and composition operators on homogeneous Lie groups. We prove that a measurable homeomorphism $\varphi: \Omega \to\widetilde{\Omega}$ belongs to the Sobolev space $L^{1}_{q}(\Omega; \widetilde{\Omega})$, $1\leq q < \infty$, if and only if $\varphi$ generates a bounded composition operator on Sobolev spaces.
\end{abstract}

\maketitle


\section{Introduction }

The composition operators on Sobolev spaces, generated by the composition rule $\varphi^\ast(f) = f \circ \varphi$, trace back to the Sobolev embedding theory \cite{GS82,M69,VGR79} and have applications in the spectral theory of elliptic operators \cite{GPU18, GU17}. In \cite{U93, VU02} (see also \cite{V88} for the case $p=q$), necessary and sufficient conditions were obtained for mappings of Euclidean domains $\varphi: \Omega \to \widetilde{\Omega}$ that generate bounded composition operators on Sobolev spaces:
$$
\varphi^{\ast}: L^1_p(\widetilde{\Omega}) \to L^1_q(\Omega), \quad 1 \leq q \leq p < \infty.
$$

The composition operators on Sobolev spaces can be considered in the framework of composition operators on partially ordered vector spaces, and in particular in the context of composition operators on Lebesgue spaces \cite{H50}. 

In the last decade, the theory of composition operators on Sobolev spaces in $\mathbb R^n$ has continued to develop intensively; see, for example, \cite{MU24_1, MU24_2}, where refined capacitory characterizations of mappings generating composition operators on Sobolev spaces, as well as their connections with mappings defined via moduli inequalities \cite{MRSY}, were obtained. The study of composition operators on weighted Sobolev spaces \cite{GU09} was further extended in \cite{V20}, where several refined results were presented.

The limit case $1 \leq q < p = \infty$, which has significant applications in nonlinear elasticity theory, particularly in the context of Ball's classes \cite{GU24, V12}, was considered in \cite{GU10M, GU10}. In the present article, we consider composition operators on Sobolev spaces in this limit case $1 \leq q < p = \infty$ on homogeneous Lie groups $\mathbb{H}$. In this setting, the main technical challenge is that, on homogeneous Lie groups, Sobolev mappings cannot be characterized in terms of their coordinate functions. We address this by introducing a method based on the upper gradient, interpreted as the volume derivative of a countably additive set function.

The composition operators on Sobolev spaces defined on homogeneous Lie groups $\mathbb{H}$  \cite{FS} were considered in \cite{U11,VU98, VU04, VU05} within the framework of the embedding theory of Sobolev spaces defined in non-holonomic metrics \cite{Ho67,VCh1,VCh2}.
The composition operators on $\mathbb{H}$  within the context of the quasiconformal mappings theory were studied in  \cite{VE15,VE22}.

The main result of \cite{VU98} can be formulated as follows.

\begin{thm} 
\label{thm:98}
Let $\varphi: \Omega \to \widetilde{\Omega}$ be a homeomorphism of domains $\Omega, \widetilde{\Omega} \subset \mathbb{H}$. 
Then $\varphi$ generates a bounded composition operator 
$$
\varphi^*: L^{1}_{p}(\widetilde{\Omega}) \to L^{1}_{q}(\Omega), \quad \varphi^*(f) = f \circ \varphi,
$$
if and only if $\varphi \in L^1_{q,\loc}(\Omega; \widetilde{\Omega})$, has finite distortion, and 
\begin{align}
&K_{p}^{p}(\varphi; \Omega) = K_{p,p}^{p}(\varphi; \Omega) = \ess\sup\limits_{\Omega} \frac{|D_H\varphi(x)|^p}{|J(x,\varphi)|} < \infty 
\quad \text{if} \quad 1 \leq q = p < \infty,\\
&K_{p,q}^{\frac{pq}{p-q}}(\varphi; \Omega) = \int\limits_{\Omega} \left( \frac{|D_H\varphi(x)|^p}{|J(x,\varphi)|} \right)^{\frac{q}{p-q}} \, dx < \infty 
\quad \text{if} \quad 1 \leq q < p < \infty.
\end{align}
\end{thm}

The theory of composition operators on homogeneous Lie groups was extended in \cite{E15}, where the non-homeomorphic case was considered. In the recent work \cite{VP24}, following \cite{U00}, sequences of mappings generating composition operators in Sobolev spaces on Carnot groups were studied.

In this paper, we establish a characterization of composition operators acting on Sobolev spaces in the case $1 \leq q < p = \infty$. Before stating our main theorem, recall that a homeomorphism $\varphi$ is called measurable in the sense of \cite{Z69} if the preimage of every Lebesgue measurable set is Lebesgue measurable. In contrast, for the case $1 \leq q \leq p < \infty$, \cite{VU98} adopted the notion of measurability with respect to the Sobolev capacity: the preimage of any $p$-capacity measurable set is $q$-capacity measurable. The following theorem presents the main result of this article.

\begin{thm} 
\label{sobolev}
Let $\Omega$, $\widetilde{\Omega}$ be domains on the  homogeneous Lie group $\mathbb H$. Then a measurable homeomorphism $\varphi$
 generates a bounded composition operator
$$
\varphi^*: L^{1}_{\infty}(\widetilde\Omega) \to L^{1}_{q}(\Omega), \quad 1\leq q < \infty,
$$
if and only if $\varphi$ belongs to the Sobolev space $L^{1}_{q}(\Omega; \widetilde{\Omega})$.
\end{thm}

This result can be considered in the context of definitions of Sobolev classes on homogeneous Lie groups.
In \cite{VU98}, it was introduced a definition of Sobolev mappings on homogeneous Lie groups, by using the upper gradient approach combined with composition with Lipschitz functions. This definition can be viewed within the framework of Reshetnyak's definition \cite{R97} and in the broader context of previous works \cite{A90, KS}. The theory of Sobolev spaces on metric measure spaces, based on the upper gradient approach, is presented in \cite{HKST}. However, the application of the upper gradient approach to Sobolev mappings taking values in metric measure spaces, in combination with the Kuratowski embedding, can be problematic, as indicated by the results in the work \cite{GE25}.

The functional definition of Sobolev classes can be seen, in some sense, as a reformulation of Reshetnyak's definition \cite{R97} in a global context of the geometric theory of composition operators on Sobolev spaces \cite{VU04,VU05}. The description of Sobolev classes within the framework of the composition operator theory is one of the approaches explored in the study of Sobolev spaces \cite{GU10M,GU10}. Furthermore, this approach facilitates the application of methods from linear operator theory to Sobolev analysis, as demonstrated in \cite{GPU24-N,GU24,MV}.

The essential tool emphasized in this article is the change of variables formula for Sobolev mappings on homogeneous Lie groups \cite{V00, VU96}, along with the use of countably additive set functions related to the norms of composition operators in Sobolev spaces \cite{U93}.

We remark that the functional definition of Sobolev classes in terms of composition operators can be generalized within the framework of Sobolev space theory on metric measure spaces, employing the methods developed in \cite{MU}.

\section{Homogeneous Lie groups and Sobolev spaces}

\subsection{Homogeneous Lie groups}

Recall that a homogeneous  Lie group \cite{FS}, or, in another terminology, a Carnot group \cite{Pa} is a~connected simply
connected nilpotent Lie group~$\mathbb H$ whose Lie algebra $V$ is decomposed into the direct sum~ $V_1\oplus\cdots\oplus V_m$, $\dim V_1\geqslant 2$, of vector spaces such that
$$
[V_1,\ V_i]=V_{i+1},\,\,1\leqslant i\leqslant m-1, \,\,\text{and}\,\,[V_1,\ V_m]=\{0\}.
$$
Let
$X_{11},\dots,X_{1n_1}$ be left-invariant basis vector fields of
$V_1$. Since they generate $V$, for each $i$, $1<i\leqslant m$, one
can choose a basis $X_{ik}$ in $V_i$, $1\leqslant k\leqslant
n_i=\dim V_i$, consisting of commutators of order $i-1$ of fields
$X_{1k}\in V_1$. We identify elements $h$ of $\mathbb H$ with
vectors $x\in\mathbb R^N$, $N=\sum_{i=1}^{m}n_i$, $x=(x_{ik})$,
$1\leqslant i\leqslant m$, $1\leqslant k\leqslant n_i$ by means of
exponential map $\exp(\sum x_{ik}X_{ik})=h$. Dilations $\delta_t$
defined by the formula 
\begin{multline}
\nonumber
\delta_t x= (t^ix_{ik})_{1\leqslant i\leqslant m,\,1\leqslant k\leqslant n_i}\\
=(tx_{11},...,tx_{1n_1},t^2x_{21},...,t^2x_{2n_2},...,t^mx_{m1},...,t^mx_{mn_m}),
\end{multline}
are automorphisms of
$\mathbb H$ for each $t>0$. Lebesgue measure $dx$ on $\mathbb R^N$
is the bi-invariant Haar measure on~ $\mathbb H$ (which is generated
by the Lebesgue measure by means of the exponential map), and
$d(\delta_t x)=t^{\nu}~dx$, where the number
$\nu=\sum_{i=1}^{m}in_i$ is called the homogeneous dimension of the
group~$\mathbb H$. 

Recall that a homogeneous norm on the group $\mathbb H$ is a continuous function
$|\cdot|:\mathbb H\to [0,\infty)$ that is $C^{\infty}$-smooth on $\mathbb
H\setminus\{0\}$ and has the following properties:

(a) $|x|=|x^{-1}|$ and $|\delta_t(x)|=t|x|$;

(b) $|x|=0$ if and only if $x=0$;

(c) there exists a constant $\tau_0>0$ such that $|x_1
x_2|\leqslant \tau_0(|x_1|+|x_2|)$ for all $x_1,x_2\in \mathbb H$.

The homogeneous norm on the group $\mathbb H$ defines a homogeneous (quasi)metric
$$
\rho(x,y)=|y^{-1} x|.
$$

Recall that a continuous map $\gamma: [a,b]\to\mathbb H$ is called a continuous curve on $\mathbb H$. This continuous curve is rectifiable if
$$
\sup\left\{\sum\limits_{k=1}^m|\left(\gamma(t_{k})\right)^{-1}\gamma(t_{k+1})|\right\}<\infty,
$$
where the supremum is taken over all partitions $a=t_1<t_2<...<t_m=b$ of the segment $[a,b]$.
In \cite{Pa} it was proved that any rectifiable curve is differentiable almost everywhere and $\dot{\gamma}(t)\in V_1$: there exists measurable functions $a_i(t)$, $t\in (a,b)$ such that
$$
\dot{\gamma}(t)=\sum\limits_{i=1}^n a_i(t)X_i(\gamma(t))\,\,\text{and}\,\,
\left|\left(\gamma(t+\tau)\right)^{-1}\gamma(t)exp(\dot{\gamma}(t)\tau)\right|=o(\tau)\,\,\text{as}\,\,\tau\to 0
$$
for almost all $t\in (a,b)$.
The length $l(\gamma)$ of a rectifiable curve $\gamma:[a,b]\to\mathbb H$ can be calculated by the formula
$$
l(\gamma)=\int\limits_a^b {\left\langle \dot{\gamma}(t),\dot{\gamma}(t)\right\rangle}_0^{\frac{1}{2}}~dt=
\int\limits_a^b \left(\sum\limits_{i=1}^{n}|a_i(t)|^2\right)^{\frac{1}{2}}~dt,
$$
where ${\left\langle \cdot,\cdot\right\rangle}_0$ is the inner product on $V_1$. The result of \cite{CH,R38} implies that one can connect two arbitrary points $x,y\in \mathbb H$ by a rectifiable curve. Remark that this rectifiable curve can be represented as a basis horizontal $l$-broken line, where $l$ does not exceed some constant $M_H<\infty$, for any two points $x,y\in \mathbb H$ \cite{FS}. In \cite{ABB} it was proved more precise estimate, that $M_H$ does not exceed $2N$. Recently, this result was refined in the case of two-step Carnot groups with horizontal distribution of corank $1$ \cite{GZ24}.

The Carnot-Carath\'eodory distance $d(x,y)$ is the infimum of the lengths of all rectifiable curves in $\mathbb{H}$ joining $x$ to $y$. 
It is well known that the Hausdorff dimension of $(\mathbb{H}, d)$ coincides with its homogeneous dimension~$\nu$; 
see, for example, \cite[Corollary~3.1.3]{ABB}.

\subsection{Sobolev spaces on homogeneous Lie groups}

Let $\mathbb H$ be a homogeneous Lie group  with the one-parameter dilatation
group $\delta_t$, $t>0$, and a homogeneous norm $|\cdot|$. Suppose that $E$ is a measurable subset of $\mathbb H$. The Lebesgue space $L_p(E)$, $p\in [1,\infty]$, is the space of $p$-th power
integrable functions $f:E\to\mathbb R$ with the standard norm:
$$
\|f\|_{L_p(E)}=\biggl(\int\limits_{E}|f(x)|^p~dx\biggr)^{\frac{1}{p}},\,\,1\leq p<\infty,
$$
and $\|f\|_{L_{\infty}(E)}=\esssup_{E}|f(x)|$ for $p=\infty$. We
denote by $L_{p,\loc}(E)$ the space of functions
$f: E\to \mathbb R$ such that $f\in L_p(F)$ for each compact
subset $F$ of $E$.

Let $\Omega$ be an open set in $\mathbb H$. The (horizontal) Sobolev space
$W^{1}_{p}(\Omega)$, $1\leqslant p\leqslant\infty$, consists of functions
$f:\Omega\to\mathbb R$ which are locally integrable in $\Omega$, possess weak
derivatives $X_{1i} f$ along the horizontal vector fields $X_{1i}$, $i=1,\dots,n_1$,
with the norm
$$
\|f\|_{W^{1}_{p}(\Omega)}=\|f\|_{L_p(\Omega)}+\|\nabla_{\textrm{H}} f\|_{L_p(\Omega)}<\infty,
$$
where $\nabla_\textrm{H} f=(X_{11}f,\dots,X_{1n_1}f)$ is the horizontal subgradient of $f$.
If $f\in W^{1}_{p}(U)$ for each bounded open set $U$ such that
$\overline{U}\subset\Omega$, then  $f$ belongs to the
class $W^{1}_{p,\loc}(\Omega)$.

The (horizontal) homogeneous Sobolev space
$L^{1}_{p}(\Omega)$, $1\leqslant p\leqslant\infty$, consists of functions
$f:\Omega\to\mathbb R$ which are locally integrable in $\Omega$, possess weak
derivatives $X_{1i} f$ along the horizontal vector fields $X_{1i}$, $i=1,\dots,n_1$,
with the seminorm
$$
\|f\|_{L^{1}_{p}(\Omega)}=\|\nabla_{\textrm{H}} f\|_{L_p(\Omega)}<\infty.
$$

\subsection{$P$-differentiability on homogeneous Lie groups}

Let $\Omega$ be a domain on the  homogeneous Lie group $\mathbb H$. Suppose that a mapping $\varphi : \Omega\to\mathbb H$. Then a Lie group homomorphism $\psi:\mathbb H\to\mathbb H$, such that $\exp^{-1}\circ\psi\circ \exp(V_1)\subset V_1$, is called the $P$-differential of $\varphi$ at the point $a$ of the domain $\Omega$ if the set
$$
A_{\varepsilon} =\{z\in E:
d(\psi(a^{-1}x)^{-1}\cdot \varphi(a)^{-1}\varphi(x))<\varepsilon d(a^{-1}x)\}
$$
is a neighborhood of $a$ (relative to $\Omega$) for every $\varepsilon>0$.

The concept of a $P$-differentiability was introduced in \cite{Pa}, where it was proved that Lipschitz mappings defined on open subsets of homogeneous Lie groups are $P$-differentiable almost everywhere. However, this result does not lead to the Rademacher-Stepanov theorem on homogeneous Lie groups, as these non-com\-mu\-ta\-tive groups lack the Kirszbraun type theorems on the extension of Lipschitz mappings from closed subsets of $\mathbb{H}$. The Rademacher-Stepanov theorem on homogeneous Lie groups $\mathbb{H}$ was obtained in \cite{VU96}, where it was proved that Lipschitz mappings defined on measurable subsets $E\subset \mathbb{H}$ are $P$-differentiable almost everywhere.
The $P$-differentiability of Lipschitz functions on measurable subsets $E\subset \mathbb{H}$ is defined as the (approximate) differentiability at density points of $E$ with respect to $E$ \cite{VU96}.
 Namely, the main result of \cite{VU96} can be formulated in the following form.

\begin{thm}
\label{RS}
Let $E$ be a measurable set on the homogeneous Lie group $\mathbb H$. 
Suppose a mapping $\varphi:E\to\mathbb{H}$ is such that 
$$
\limsup_{y\to x, y\in E} \frac{d(\varphi(x),\varphi(y))}{d(x,y)}<\infty,\,\,\text{for almost all}\,\,x\in E.
$$

Then the $P$-differential of the mapping $\varphi$ exists and is unique at almost every point $x\in E$. The corresponding homomorphism of the Lie algebras $D_H\varphi(x)$ takes the basis vectors $X_{11},\dots,X_{1n_1}$ of the subalgebra $V_1$ into the vectors $\{X_{1i} \in V_1\}$, $i=1,\dots,n_1$.

\end{thm}

In the proof of Theorem~\ref{RS} in \cite{VU96}, the representation of any point $h\in\mathbb H$
$$
h=\prod\limits_{j=1}^l\exp\left(a_j X_{i_j}\right), \,\,i_j=1,...,n_1,\,\,l\leq M_H,
$$
where $X_{i_j}\in V_1$, has a significant role. This representation states that any two points on homogeneous Lie groups $\mathbb{H}$ can be connected by a basis horizontal $l$-broken line, where $l$ does not exceed the bound $M_H$  \cite{ABB,FS}.

It is worth noting that the results and the method of differentiability at density points, presented in \cite{VU96}, which address the Rademacher-Stepanov theorem, appear to have been reiterated in \cite{VM01}. Furthermore, the proof of the Rademacher-Stepanov theorem, which can be realized on Carnot-Carath\'eodory spaces, was provided in \cite{V00}

\subsection{Sobolev mappings on homogeneous Lie groups}

Let $\Omega$ be a domain on the homogeneous Lie group $\mathbb H$. Then a mapping  $\varphi:\Omega\to\mathbb{H}$ is absolutely continuous on lines ($\varphi\in \ACL(\Omega;\mathbb{H})$), if for each domain $U$ such that $\overline{U}\subset\Omega$ and each foliation $\Gamma_i$ defined by a left-invariant vector field $X_{1i}$, $i=1,\dots,n_1$, $\varphi$ is absolutely continuous on $\gamma\cap U$ with respect to one-dimensional Hausdorff measure for $d\gamma$-almost every curve $\gamma\in\Gamma_i$. Recall that the measure $d\gamma$ on the foliation $\Gamma_i$ equals  the inner product $i(X_{1i})dx$ of the vector field $X_{1i}$ and the bi-invariant volume $dx$ (see, for example, \cite{Fe,VU96}).

Since $X_{1i} \varphi(x)\in{V}_1$ for almost all $x\in\Omega$ \cite{Pa},
$i=1,\dots,n_1$, the linear mapping $D_H \varphi(x)$ with matrix
$(X_{1i}\varphi_j(x))$, $i,j=1,\ldots,n_1$, takes the horizontal subspace $V_1$ to $V_1$
and is called the formal
horizontal differential of the mapping $\varphi$ at $x$. Let $|D_H
\varphi(x)|$ be its norm:
$$
|D_H\varphi(x)|=\sup\limits_{\xi\in
V_1,\,|\xi|=1} |D_H\varphi(x)(\xi)|.
$$

It was proved \cite{V00,VU96} that this formal horizontal differential $D_H:V_1\to{V}_1$ induces a homomorphism
$D\varphi:V\to{V}$ of the Lie algebras which is called the formal differential, and 
\begin{equation}
\label{dif}
|D\varphi(x)|\leq C_H |D_H\varphi(x)|,\,\,\text{for almost all}\,\,x\in\Omega,
\end{equation}
where $C_H$ is a constant with $0<C_H<\infty$.

The determinant of the matrix
$D\varphi(x)$ is called the (formal) Jacobian of the mapping $\varphi$, it is denoted by $J(x,\varphi)$.

We will use the definition of Sobolev measurable mappings (a mapping whose preimage of any Lebesgue measurable set is Lebesgue measurable) in terms of composition operators on spaces of Lipschitz functions \cite{V99,VU98}, which generalize the definition of Sobolev classes defined in $\mathbb R^n$, given in \cite{R97}. Recall that 
the Sobolev class $L^{1}_{p}(\Omega;\mathbb{H})$, $1 \leq p<\infty$, is a class of measurable mappings $\varphi:\Omega\to\mathbb{H}$ such that:

\noindent
$(1)$ for every point  $z \in \mathbb{H}$ the function $[\varphi]_{z}(x) = d(\varphi(x), z)$ belongs to the Sobolev space $L^{1}_{p}(\Omega)$;

\noindent
$(2)$ there exists a function (\textit{an upper gradient of mapping $\varphi$}) $g \in L_p(\Omega)$ that is an upper gradient of $[\varphi]_{z}$ for all $z \in \mathbb{H}$, i.e. there exists a constant $K_p$ such that
$$
|\nabla_H([\varphi]_{z})|(x)\leqslant K_p\cdot g(x), \,\, \text{for almost all}\,\,x\in\Omega.
$$

The Sobolev class $L^{1}_{p,\loc}(\Omega;\mathbb{H})$, $1 \leq p<\infty$, is a class of measurable mappings $\varphi:\Omega\to\mathbb{H}$ such that $\varphi\in L^{1}_{p}(U;\mathbb{H})$, where $U$ is a bounded open set, $\overline{U}\subset \Omega$.

\begin{rem}
To ensure the validity of the chain rule when composing Sobolev mappings with Lipschitz functions, we define Sobolev classes via the composition of measurable mappings -- those whose preimage of a Lebesgue measurable set is also Lebesgue measurable. This approach is crucial in generalizing Sobolev mappings, particularly in the context of Carnot groups.
\end{rem}

Let us recall the formula of the change of variables in the Lebesgue integral \cite{V00, VU96}.
Let $\varphi : \Omega \to \mathbb{H}$ be a measurable mapping that belongs to the Sobolev class $W^1_{1,\loc}(\Omega,\mathbb{H})$.
Then there exists a measurable set $S \subset \Omega$, $|S| = 0$, such that the mapping $\varphi : \Omega \setminus S \to \mathbb{H}$ has the Luzin $N$-property (the image of a set of measure zero has measure zero) and the change of variables formula
\begin{equation}
\label{chvf}
\int\limits_E f\circ\varphi (x) |J(x,\varphi)|~dx = \int\limits_{\mathbb{H} \setminus \varphi(S)} f(y) N_\varphi(y,E)~dy
\end{equation}
holds for every measurable set $E \subset \Omega$ and every non-negative measurable function $f : \mathbb{H} \to \mathbb{R}$. Here 
$N_\varphi(y,E)$ is the multiplicity function defined as the number of preimages of $y$ under $\varphi$ in $E$.
If $\varphi$ possesses the Luzin $N$-property, then $|\varphi(S)| = 0$, and the second integral can be rewritten as an integral over $\mathbb{H}$.

\section{Composition operators on Sobolev spaces}

In the geometric theory of composition operators on Sobolev spaces \cite{VU04,VU05}, a critical role is played by set functions associated with the norms of these operators. The set function~$\Phi$ was first introduced in \cite{U93} in connection with the solution of Reshetnyak's problem, formulated in 1968 at the First Donetsk Colloquium on Mapping Theory.

Recall that a locally finite non-negative function $\Phi$ defined on open subsets of $\Omega \subset \mathbb{H}$ is called a monotone countably additive set function \cite{VU04,VU05} if

\noindent
1) $\Phi(U_1)\leq \Phi(U_2)$ if $U_1\subset U_2\subset\Omega$;

\noindent
2)  for any collection $U_i \subset \Omega$, $i=1,2,...$, of mutually disjoint open sets
$$
\sum_{i=1}^{\infty}\Phi(U_i) = \Phi\left(\bigcup_{i=1}^{\infty}U_i\right).
$$

The following lemma gives properties of monotone countably additive set functions defined on open subsets of $\Omega\subset \mathbb H$ \cite{VU04,VU05}.

\begin{lem}
\label{lem:AddFun}
Let $\Phi$ be a monotone countably additive set function defined on open subsets of the domain $\Omega\subset \mathbb H$. Then

\noindent
(a) at almost all points $x\in \Omega$ there exists a finite derivative
$$
\lim\limits_{r\to 0}\frac{\Phi(B(x,r))}{|B(x,r)|}=\Phi'(x);
$$

\noindent
(b) $\Phi'(x)$ is a measurable function;

\noindent
(c) for every bounded open set $U\subset \Omega$ the inequality
$$
\int\limits_U\Phi'(x)~dx\leq \Phi(U)
$$
holds.
\end{lem}

We use the following countable additive property of norms of composition operators on Sobolev spaces \cite{VU04,VU05}.  Let $\Omega$, $\widetilde{\Omega}$ be domains on the homogeneous Lie group $\mathbb H$ and let $\varphi: \Omega \to \widetilde{\Omega}$ be a homeomorphism.  Define the set function $\Phi(\widetilde{A})$, $\widetilde{A}\subset \widetilde{\Omega}$ are open bounded subsets, by the rule
\begin{equation}
\label{set_fun}
\Phi(\widetilde{A}) = \sup\limits_{f \in L^{1}_{\infty}(\widetilde{A}) \cap C_0(\widetilde{A})}\left( \frac{\|\varphi^\ast(f)\|_{L^{1}_{q}(\Omega)}}{\|f\|_{L^{1}_{\infty}(\widetilde{A})}} \right)^{q},
\end{equation}
where $C_0(\widetilde{A})$ is the space of continuous functions supported in $\widetilde{A}$.  This set function $\Phi(\widetilde{A})$ is a minor modification of the set functions considered in \cite{VU04,VU05}, in that here we take the supremum over continuous functions 
$f \in L^1_{\infty}(\widetilde{A}) \cap C_0(\widetilde{A})$.

\begin{lem}
\label{mainlem}
Let $\Omega$, $\widetilde{\Omega}$ be domains on the  homogeneous Lie group $\mathbb H$. Suppose that a homeomorphism $\varphi: \Omega \to \widetilde{\Omega}$ generates a bounded composition operator
$$
\varphi^*: L^{1}_{\infty}(\widetilde\Omega) \to L^{1}_{q}(\Omega), \quad \varphi^*(f) = f \circ \varphi, \quad 1\leq q < \infty.
$$
Then the function $\Phi(\widetilde{A})$ defined by \eqref{set_fun} is a bounded monotone countable additive set function defined on open bounded subsets $\widetilde{A} \subset \widetilde\Omega$.
\end{lem}

\begin{proof}
If $\widetilde{A}_1\subset \widetilde{A}_2$ are bounded open  subsets in $\widetilde{\Omega}$, then, extending
the functions of the space $C_0(\widetilde{A}_1)$ by zero on the set $\widetilde{A}_2$, we have the inclusion
$$
\big(L^{1}_{\infty}(\widetilde{A}_1)\cap C_0(\widetilde{A}_1)\big)\subset
\big(L^{1}_{\infty}(\widetilde{A}_2)\cap C_0(\widetilde{A}_2)\big).
$$
Hence we obtain the monotonicity of the set function $\Phi$:
\begin{multline}
\Phi(\widetilde{A}_1)=\sup\limits_{f\in
L^{1}_{\infty}(\widetilde{A}_1)\cap C_0(\widetilde{A}_1)} \Biggl(
\frac{\bigl\|\varphi^{\ast} (f)\bigr\|_{{L}_{q}^{1}(\Omega)}}
{\bigl\|f\bigr\|_{L^{1}_{\infty}(\widetilde{A}_1)}}
\Biggr)^{q}\\
\leq\sup\limits_{f\in
L^{1}_{\infty}(\widetilde{A}_2)\cap C_0(\widetilde{A}_2)} \Biggl(
\frac{\bigl\|\varphi^{\ast} (f)\bigr\|_{{L}_{q}^{1}(\Omega)}}
{\bigl\|f\bigr\|_{L^{1}_{\infty}(\widetilde{A}_2)}}
\Biggr)^{q}=\Phi(\widetilde{A}_2).
\nonumber
\end{multline}

Now, let $\widetilde{A}_{i}$, $i=1,2,...$, be open disjoint subsets of $\widetilde{\Omega}$ and denote 
$\widetilde{A}_0=\bigcup\limits_{i=1}^{\infty}\widetilde{A}_i$.
For $i=1,2,...$ we consider functions $f_i\in L_{\infty}^{1}(\widetilde{A}_i)\cap C_0(\widetilde{A}_i)$ which satisfy the following conditions
$$
\bigl\|\varphi^{\ast} f_i\bigr\|^q_{{L}_{q}^{1}(\Omega)}\geq \bigl(\Phi(\widetilde{A}_i)\bigl(1-{\frac{\varepsilon}{2^i}}\bigr)\bigr)\bigl\|f_i\bigr\|^q_{L_{\infty}^{1}(\widetilde{A}_i)}
$$
and
$$
\bigl\|f_i\bigr\|^q_{L^{1}_{\infty}(\widetilde{A}_i)}=1,\,\,i=1,2,...,
$$
where $\varepsilon\in(0,1)$ is a fixed number.
Letting $f_N=\sum\limits_{i=1}^{N}f_i$ and $A_i=\varphi^{-1}(\widetilde{A}_i)$, we obtain
\begin{multline}
\bigl\|\varphi^{\ast} \left(f_N\right) \bigr\|^q_{{L}_{q}^{1}(\Omega)} \geq \sum\limits_{i=1}^{N} \bigl\|\varphi^{\ast} (f_i) \bigr\|^q_{{L}_{q}^{1}(A_i)} \\ \geq
\sum\limits_{i=1}^{N}\left(\Phi(\widetilde{A}_i)\left(1-{\frac{\varepsilon}{2^i}}\right)\right)
\bigl\|f_i\bigr\|_{L_{\infty}^{1}(\widetilde{A}_i)}^q
=
\sum\limits_{i=1}^{N}\Phi(\widetilde{A}_i)
\left(1-{\frac{\varepsilon}{2^i}}\right)
\bigl\|f_N\bigr\|^q_{L_{\infty}^{1} \Bigl(\bigcup\limits_{i=1}^{N}\widetilde{A}_i\Bigr)}
\\
\geq
\biggl(\sum\limits_{i=1}^{N}\Phi(\widetilde{A}_i)
-\varepsilon\Phi(\widetilde{A}_0) \biggr)
\bigl\|f_N\bigr\|^q_{L_{\infty}^{1} \Bigl(\bigcup\limits_{i=1}^{N}\widetilde{A}_i\Bigr)},
\nonumber
\end{multline}
since the sets, on which the gradients $\nabla\varphi^{\ast} (f_i)$
do not vanish, are disjoint. Hence, it follows that
$$
\Phi(\widetilde{A}_0)^{\frac{1}{q}}\geq\sup\frac
{\bigl\|\varphi^{\ast} g_N\mid{L}_{q}^{1} (\Omega)\bigr\|}
{\biggl\|g_N\mid L_{\infty}^{1}
\Bigl(\bigcup\limits_{i=1}^{N}\widetilde{A}_i\Bigr)\biggr\|}\geq
\biggl(\sum\limits_{i=1}^{N}\Phi(\widetilde{A}_i)-\varepsilon\Phi(\widetilde{A}_0)
\biggr)^{\frac{1}{q}},
$$
where the least upper bound is taken over all the
above-mentioned functions $g_N\in L_{\infty}^{1}
\Bigl(\bigcup\limits_{i=1}^{N}\widetilde{A}_i\Bigr)\cap C_0\Bigl(\bigcup\limits_{i=1}^{N}\widetilde{A}_i\Bigr)$. Since both
$N$ and $\varepsilon$ are arbitrary, we have
$$
\sum\limits_{i=1}^{\infty}\Phi(\widetilde{A}_i)
\leq \Phi\Bigl(\bigcup\limits_{i=1}^{\infty}\widetilde{A}_i\Bigr).
$$
The validity of the inverse inequality can be proved in a straightforward manner.
\end{proof}

Let $\varphi: \Omega \to \widetilde{\Omega}$ be a homeomorphism between domains $\Omega, \widetilde{\Omega} \subset \mathbb{H}$.  
The volume derivative of $\varphi$ at a point $x \in \Omega$ is defined by  

\[
\varphi'_v(x) = \lim_{r \to 0} \frac{|\varphi(B(x,r))|}{|B(x,r)|}.
\]

Since $\varphi$ is a homeomorphism, the set function $\Phi(A) := |\varphi(A)|$ is monotone and countably additive on the Borel subsets of $\Omega \subset \mathbb{H}$.  Therefore, by Lemma~\ref{lem:AddFun} (see also \cite{VU04,VU05}), the volume derivative $\varphi'_v$ is finite for almost every $x \in \Omega$ and belongs to the Lebesgue space $L_{1,\mathrm{loc}}(\Omega)$.

Let us formulate the change of variable formula in the Lebesgue integral \cite{VU04,VU05} in the case of homeomorphisms with the volume derivative.

\begin{thm}
Let $\varphi: \Omega \to \widetilde{\Omega}$ be a homeomorphism of domains $\Omega,\widetilde{\Omega}\subset \mathbb H$. 
Then there exists a Borel set $S\subset \Omega$, $|S|=0$ such that  the mapping $\varphi:\Omega\setminus S \to \widetilde{\Omega}$ has the Luzin $N$-property and the change of variables formula
\begin{equation}
\label{chv}
\int\limits_E f\circ\varphi (x) \varphi'_v(x)~dx=\int\limits_{\varphi(E\setminus S)} f(y)~dy
\end{equation}
holds for every measurable set $E\subset \Omega$ and every non-negative measurable function $f: \widetilde{\Omega}\to\mathbb R$.
\end{thm}

In the present article we give the analytical characterization of measurable homeomorphisms $\varphi: \Omega \to\widetilde{\Omega}$ which generate a bounded composition operator
$$
\varphi^*: L^{1}_{\infty}(\widetilde\Omega) \to L^{1}_{q}(\Omega), \quad \varphi^*(f) = f \circ \varphi, \quad 1\leq q < \infty.
$$
Recall that the homeomorphisms  $\varphi$ is called measurable in the sense of \cite{Z69}, if a preimage of a Lebesgue measurable set is Lebesgue measurable set. It is equivalent, that the homeomorphism  $\varphi: \Omega \to\widetilde{\Omega}$ has the Luzin $N^{-1}$-property \cite{H50}. Hence, by the change of variables formula \eqref{chv}, for measurable homeomorphisms  $\varphi$ the volume derivative $\varphi'_v>0$ a.e. in $\Omega$.

In the next assertion we recall Theorem~\ref{sobolev}~from the Introduction, 
which characterizes Sobolev mappings in terms of composition operators 
on spaces of Lipschitz functions.

\medskip
\noindent
\textbf{Theorem~\ref{sobolev}.}
\textit{Let $\Omega$, $\widetilde{\Omega}$ be domains on the homogeneous Lie group $\mathbb{H}$. 
Then a measurable homeomorphism $\varphi: \Omega \to \widetilde{\Omega}$ generates a bounded composition operator
\[
\varphi^*: L^{1}_{\infty}(\widetilde{\Omega}) \to L^{1}_{q}(\Omega), 
\quad \varphi^*(f) = f \circ \varphi, 
\quad 1 \leq q < \infty,
\]
if and only if $\varphi \in L^{1}_{q}(\Omega; \widetilde{\Omega})$.}

\begin{proof} \textit{Necessity.}
Suppose $z \in \widetilde{\Omega}$ be an arbitrary basepoint. 
The functions $d_z(y) = d(y, z)$, $z \in \widetilde{\Omega}$, are $1$-Lipschitz, with 
\[
\nabla_{H} d_z (y) = 1 \quad \text{for almost all } y \in \widetilde{\Omega},
\]
see, for example, \cite[Remark~2.1 and \S 4]{HK00} and \cite[Example~4.1.2, \S 13.3]{HKST}, 
and belong to the Sobolev space $L^{1}_{\infty}(\widetilde{\Omega})$.

By Lemma \ref{mainlem} for every open set $\widetilde{A} \subset \widetilde{\Omega}$ and for every continuous function $f \in L^{1}_{\infty}(\widetilde{\Omega})$, $\supp f \subset \widetilde{A}$, the following inequality holds:
\begin{equation}
\label{eq1}
\|\varphi^*(f)\|_{L^{1}_{q}(\varphi^{-1}(\widetilde{A}))} \leq (\Phi(\widetilde{A}))^{\frac{1}{q}} \|f\|_{L^{1}_{\infty}(\widetilde{A})}.
\end{equation}

Define a smooth cutoff function:
\begin{equation*}
\eta(t)=
\begin{cases}
1,&\text{if}\,\, 0\leq t \leq 1, \\ 
e^{-\frac{t-1}{2-t}},&\text{if}\,\, 1 < t < 2, \\ 
0,&\text{if}\,\, 2\leq t<\infty.
\end{cases}
\end{equation*}

Fix an arbitrary point $y_0 \in \widetilde{\Omega}$. We denote by $B(y_0, r)\subset\widetilde{\Omega}$ a ball with a center at the point $y_0$ and the given radius $r>0$. Since $\varphi$ is the homeomorphism, $\varphi^{-1}(B(y_0, r))$ is an open connected set.    
We consider test functions defined by the equation:
$$
f_z(y) = (d_z(y) -d_z(y_0)) \eta\left(\frac{d(y,y_0)}{r}\right),\quad z \in \widetilde{\Omega}.
$$

Then, for the function $[\varphi]_{z}(x) = d(\varphi(x), z)$, we have 
$(f_z \circ \varphi)(x) = [\varphi]_{z}$ for all $x \in \varphi^{-1}(B)$ and $|\nabla_{H} (f_z \circ \varphi)| = |\nabla_{H} [\varphi]_{z}|$ for almost all $x \in \varphi^{-1}(B)$. 

The following estimates of $\nabla_{H}f_z$ were obtained in \cite{VU98}:
\begin{multline*}
|\nabla_{H}f_z (y)| \leq 
\left| \nabla_{H}\big(d_z(y)-d_z(y_0)\big) \cdot 
       \eta\!\left(\frac{d(y,y_0)}{r}\right) \right| \\
+ \left| d_z(y)-d_z(y_0) \right| \cdot 
  \left| \eta'\!\left(\frac{d(y,y_0)}{r}\right) \right| \cdot 
  \frac{\left| \nabla_{H} d(y,y_0) \right|}{r} \\
= \left| \nabla_{H}d_z(y) \right| \cdot 
  \left| \eta\!\left(\frac{d(y,y_0)}{r}\right) \right| 
+ \frac{\left| d_z(y)-d_z(y_0) \right|}{r} \cdot 
  \left| \eta'\!\left(\frac{d(y,y_0)}{r}\right) \right|.
\end{multline*}

Since $d_z$ are $L$-Lipschitz functions, with $L=1$ by the reverse triangle inequality, we have 
$$
|(d_z(y)-d_z(y_0))|\leq d(y,y_0)\leq r.
$$
Hence 
$$
|\nabla_{H}f_z (y)| \leq \left|\eta\left(\frac{d(y,y_0)}{r}\right)\right| 
+ \left|\eta'\left(\frac{d(y,y_0)}{r}\right)\right|
\leq 1+\left|\eta'\left(\frac{d(y,y_0)}{r}\right)\right|
$$
for almost all $y\in B(y_0,r)$.

Since $\eta \in C^\infty_0(\mathbb{R})$, then the derivative $\eta'\left(\frac{d(y,y_0)}{r}\right)$ is bounded and we obtain that there exists a constant $C<\infty$, such that 
$$
|\nabla_{H}f_z (y)| \leq 1+C,
$$
for almost all $y\in B(y_0,r)$.

Hence, the inequality (\ref{eq1}) implies that
$$
\int_{\varphi^{-1}(B)} |\nabla_{H} [\varphi]_{z} (x)|^q \, dx \leq \widetilde{C}^q_{q} \cdot \Phi(2B),
$$
since $\operatorname{supp} f_z \subset 2B$.

By the change of variables formula \eqref{chv} and taking into account that the volume derivative of a measurable homeomorphism  $\varphi'_v(x)>0$ a.e. in $\Omega$, we have
\begin{multline*}
\int_{\varphi^{-1}(B)} |\nabla_{H} [\varphi]_{z} (x)|^q~dx=\int_{\varphi^{-1}(B)\setminus S} |\nabla_{H} [\varphi]_{z} (x)|^q~dx
\\
=\int_{\varphi^{-1}(B)\setminus S} \frac{|\nabla_{H} [\varphi]_{z} (x)|^q}{\varphi'_v(x)}\varphi'_v(x)~dx=
\int_{B\setminus \varphi(S)} \frac{|\nabla_{H} [\varphi]_{z} (\varphi^{-1}(y))|^q}{\varphi'_v(\varphi^{-1}(y))}~dy.
\end{multline*}
Hence
$$
\int_{B\setminus \varphi(S)} \frac{|\nabla_{H} [\varphi]_{z} (\varphi^{-1}(y))|^q}{\varphi'_v(\varphi^{-1}(y))}~dy \leq \widetilde{C}^q_{q} \cdot \Phi(2B).
$$

Further, divide both sides by the measure of the ball $|B|$ and, applying the Lebesgue Differentiation Theorem as $r \to 0$ at the density points of the set $B \setminus \varphi(S)$, we obtain:
\begin{equation}\label{est}
\frac{(\nabla_{H} [\varphi]_{z} (\varphi^{-1}(y)))^q}{\varphi'_v(\varphi^{-1}(y))} \leq 2^{\nu}\widetilde{C}^q_{q} \cdot \Phi'(y) \quad \text{for almost all } y \in \widetilde{\Omega}\setminus \varphi(S).
\end{equation}

Since $\varphi$ is a measurable homeomorphism and $|S|=0$, the above inequality equivalents to the equation
\begin{equation}
\label{eqpoint}
		(\nabla_{H} [\varphi]_{z} (x))^q \leq 2^{\nu}\widetilde{C}^q_{q} \cdot \Phi'(\varphi(x))\varphi'_v(x)<\infty,
\end{equation}
which holds for almost all  $x\in \Omega$ and for all $z\in \widetilde{\Omega}$.

Hence, 
\begin{multline*}
\int\limits_{U}(\nabla_{H} [\varphi]_{z} (x))^q~dx \leq  2^{\nu}\widetilde{C}^q_{q} \int\limits_{U} \Phi'\left(\varphi(x)\right) \varphi'_v(x)~dx \\
    \leq 2^{\nu}\widetilde{C}^q_{q} \int\limits_{\widetilde{U}} \Phi'(y)~dy\leq 2^{\nu}\widetilde{C}^q_{q}\Phi(\widetilde{U}) <\infty,
\end{multline*}
for any bounded open set $U\subset\Omega$, $\overline{U}\subset\Omega$ and $\widetilde{U}=\varphi(U)$.

Therefore, the function $g(x):=\Phi'(\varphi(x))\cdot \varphi'_v(x)$ is an upper gradient of the homeomorphism $\varphi: \Omega \to \widetilde{\Omega}$ and $\varphi$ belongs to the Sobolev class $L^{1}_{q}(\Omega; \widetilde{\Omega})$, see Lemma~2 in~\cite{VU98} for details.

\textit{Sufficiency.}  Let $f\in L^1_{\infty}(\widetilde{\Omega})$. Then, since a homeomorphism $\varphi^{-1}$ has the Luzin $N$-property, the chain rule holds, and we have
\begin{multline*}
\|\varphi^{\ast}(f)\|_{L^1_{q}({\Omega})}=\left(\int\limits_{\Omega}|\nabla_H(f\circ\varphi)(x)|^q~dx\right)^{\frac{1}{q}}\leq
\left(\int\limits_{\Omega}|D_H\varphi(x)|^q|(\nabla_H f)(\varphi(x))|^q~dx\right)^{\frac{1}{q}}\\
\leq
\left(\int\limits_{\Omega}|D_H\varphi(x)|^q~dx\right)^{\frac{1}{q}}\cdot \ess\sup_{x\in\Omega}|(\nabla_H f)(\varphi(x))|(\varphi(x))=
\|\varphi\|_{L^1_q(\Omega;\widetilde{\Omega})}\cdot \|f\|_{L^1_{\infty}(\Omega)}.
\end{multline*} 

Hence, the composition operator
$$
\varphi^*: L^{1}_{\infty}(\widetilde\Omega) \to L^{1}_{q}(\Omega), \quad \varphi^*(f) = f \circ \varphi, \quad 1\leq q < \infty,
$$
is bounded.		
\end{proof}

Because in the proof of the necessity of Theorem~\ref{sobolev} we use test functions $f_z(y)$ belonging to  $L^{1}_{\infty}(\widetilde\Omega)\cap L^{1}_{p}(\widetilde\Omega)$, for any $1\leq q\leq p<\infty$, we can formulate the following assertion.

\begin{lem}
\label{sobolev2}
Let $\Omega$, $\widetilde{\Omega}$ be domains on the  homogeneous Lie group $\mathbb H$. Suppose that a measurable homeomorphism $\varphi: \Omega \to\widetilde{\Omega}$ generates a bounded composition operator
$$
\varphi^*: L^{1}_{\infty}(\widetilde\Omega)\cap L^{1}_{p}(\widetilde\Omega) \to L^{1}_{q}(\Omega), \quad \varphi^*(f) = f \circ \varphi, \quad 1\leq q\leq p < \infty.
$$
Then $\varphi$ belongs to the Sobolev space $L^{1}_{q}(\Omega; \widetilde{\Omega})$.
\end{lem}

\begin{rem}
Theorem~\ref{sobolev} can be generalized on the metric measure spaces.
\end{rem}

\section{Declarations}

\begin{itemize}

\item \textbf{Funding}: Not applicable.

\item \textbf{Conflict of interest/Competing interests}: The author declares that there is no conflict of interest.

\item \textbf{Ethics approval}: Not applicable.

\item \textbf{Availability of data and materials}: Data sharing not applicable to this article as no datasets were generated or analysed during the current study.

\item \textbf{Authors' contributions}: Not applicable. 

\end{itemize}

\vskip 0.4cm

\noindent
{\bf Acknowledgments}:
The author would like to thank A.~Agrachev for useful discussions on the Chow--Rashevskii theorem.

\vskip 0.4cm

\noindent
Alexander Ukhlov;\\ 
Department of Mathematics, \\
Ben-Gurion University of the Negev,\\ 
P.O.Box 653, Beer Sheva, 8410501, Israel. 

\noindent							
\emph{E-mail address:} \email{ukhlov@math.bgu.ac.il} \\

\end{document}